%% file: STS19.tex
\documentclass[a4paper,12pt]{article}
\usepackage{amssymb}
\usepackage{amsmath}
\usepackage{latexsym}
\usepackage{amsthm}
\usepackage{graphicx}
\usepackage{tikz}
\usepackage{array}
\usepackage{enumitem}
\usetikzlibrary{svg.path}
\usepackage[colorlinks=true,linkcolor=black,citecolor=black,urlcolor=black]{hyperref}
\usepackage[top=2cm,left=2cm,bottom=2cm,right=2cm]{geometry}
\usepackage[section]{placeins}

\newtheorem{theorem}{Theorem}

\newtheorem{lemma}[theorem]{Lemma}
\newtheorem*{lemma*}{Lemma}

\theoremstyle{definition}
\newtheorem*{remark}{Remark}

\newenvironment{config}{\vspace*{0.5ex}\tt\setlength{\parskip}{1pt}}{}

\usepackage{todonotes}

\newcolumntype{P}[1]{>{\centering\arraybackslash}p{#1}}
\newcolumntype{M}[1]{>{\centering\arraybackslash}m{#1}}

\date{}
\DeclareMathOperator{\STS}{STS}

\newcommand{\Sw}[1][]{\operatorname{S}_{#1}}
\newcommand{\DSw}[1][]{\overrightarrow{\operatorname{S}_{#1}}}

\begin{document}
	\title{Cycle switching in Steiner triple systems of order 19}
	\author{Grahame Erskine and Terry S. Griggs\\
		School of Mathematics and Statistics\\
		The Open University\\
		Walton Hall\\
		Milton Keynes MK7 6AA\\
		UNITED KINGDOM\\
		\texttt{grahame.erskine@open.ac.uk}\\
		\texttt{terry.griggs@open.ac.uk}}
	\maketitle
	
	\begin{abstract}\noindent
		Cycle switching is a particular form of transformation applied to isomorphism classes of a Steiner triple system of a given order $v$ (an $\STS(v)$), yielding another $\STS(v)$. This relationship may be represented by an undirected graph. An $\STS(v)$ admits cycles of lengths $4,6,\ldots,v-7$ and $v-3$. In the particular case of $v=19$, it is known that the full switching graph, allowing switching of cycles of any length, is connected. We show that if we restrict switching to only one of the possible cycle lengths, in all cases the switching graph is disconnected (even if we ignore those $\STS(19)$s which have no cycle of the given length). Moreover, in a number of cases we find intriguing connected components in the switching graphs which exhibit unexpected symmetries. Our method utilises an algorithm for determining connected components in a very large implicitly defined graph which is more efficient than previous approaches, avoiding the necessity of computing canonical labellings for a large proportion of the systems.
	\end{abstract}
	
	\noindent\textbf{MSC2020 classification:} 05B07.\\
	\noindent\textbf{Keywords:} Steiner triple system; cycle switching; connected graph.
	
	\section{Introduction}\label{sec:intro}
	A \emph{Steiner triple system} of \emph{order} $v$, denoted by $\STS(v)$, is an ordered pair $S=(V,\mathcal{B})$ where $V$ is a set of \emph{points} of cardinality $v$, and $\mathcal{B}$ is a collection of \emph{triples}, also called \emph{blocks}, which collectively have the property that every pair of distinct points of $V$ is contained in precisely one triple. Such systems exist if and only if $v \equiv 1$ or $3 \pmod 6$; first proved by Kirkman in 1847~\cite{KIRK}. In this paper, we take the point set $V$ to be $\{0,1,2,\ldots,v-1\}$; two such Steiner triple systems $(V,\mathcal{B})$ and $(V,\mathcal{D})$ are said to be \emph{isomorphic} if there exists a bijection $\phi$ on their point set such that every triple $B \in \mathcal{B}$ maps to a triple $\phi(B) \in \mathcal{D}$. Up to isomorphism, the numbers of $\STS(v)s$ of orders 7, 9, 13 and 15 are 1, 1, 2 and 80 respectively. With the advent of modern computing techniques, Kaski and \"{O}sterg\r{a}rd~\cite{KO} determined the number of non-isomorphic $\STS(19)$s; there are 11,084,874,829 which are now stored in compact notation on the Internet~\cite{STS19}. They have been extensively analysed with the results appearing in~\cite{EIGHT}.
	
	A \emph{switch} is an operation on an $\STS(v)$ whereby a set of blocks is removed from the system and replaced by a new set of blocks of the same cardinality: provided the removed and replacement block sets cover precisely the same set of pairs of points, the resulting structure will be a new $\STS(v)$. To avoid redundancy, we usually insist that the removed and replacement block sets be disjoint. The two sets of blocks in a switch operation are known as \emph{tradeable configurations}; together they constitute a \emph{trade}. The smallest possible tradeable configuration in a Steiner triple system is a set of four blocks containing six points in an arrangement known as a \emph{Pasch} configuration or 4-\emph{cycle}. Small tradeable configurations have been classified in~\cite{FGG2}.
	
	In general, finding particular tradeable configurations in a given $\STS(v)$ may be an onerous computational task. Most research into switching of Steiner triple systems has concentrated on a particular class of tradeable configurations known as \emph{cycles} which we describe now, following the notation of~\cite{KMO}. 
	
	Let $a,b$ be distinct points in an $\STS(v)$, and let $c$ be the unique point which shares a block with $a$ and $b$. The \emph{cycle graph} $C_{ab}$ is defined as follows. Let $X_{ab}$ be the set $V\setminus\{a,b,c\}$ and let $\mathcal{B}_a$ be the set of blocks in $\mathcal{B}$ containing $a$, apart from the block $\{a,b,c\}$; and similarly for $\mathcal{B}_b$. We define the graph $C_{ab}$ on the vertex set $X_{ab}$ to have edges $\{B\setminus\{a\}:B\in \mathcal{B}_a\}$ and $\{B\setminus\{b\}:B\in \mathcal{B}_b\}$. It is easy to see that each vertex in $C_{ab}$ has valency 2, and so $C_{ab}$ is a 2-regular graph of order $v-3$. Further, each cycle in $C_{ab}$ has even length at least 4.
	
	Now let $\{x_1,x_2,\ldots,x_{2k}\}$ be the vertices of some cycle of length $2k$ in $C_{ab}$. There are $2k$ blocks of the $\STS$ of the form:
	
	$\quad \{a,x_1,x_2\}, \{a,x_3,x_4\},\ldots,\{a,x_{2k-1},x_{2k}\} \in \mathcal{B}_a$; and
	
	$\quad \{b,x_1,x_{2k}\}, \{b,x_2,x_3\},\ldots,\{b,x_{2k-2},x_{2k-1}\} \in \mathcal{B}_b$.
	
	The \emph{cycle switch} relative to the pair $\{a,b\}$ and the set $\{x_1,x_2,\ldots,x_{2k}\}$ is formed by exchanging the roles of $a$ and $b$ in these blocks. Note that the cycle switch is uniquely determined by the pair $\{a,b\}$ and any one of the points $\{x_1,x_2,\ldots,x_{2k}\}$. For a given pair $\{a,b\}$, the cycle graph $C_{ab}$ may consist of a single Hamiltonian cycle of length $v-3$, or multiple disjoint cycles of length totalling $v-3$. In the case where $C_{ab}$ consists of exactly two cycles of length $2k$ and $2\ell$ where $2k+2\ell=v-3$, we call these two cycles a \emph{complementary pair}. The significance of Hamiltonian cycles and complementary pairs is shown by the following lemma, proved in~\cite{GGM}.
	
	\begin{lemma}\label{lem:compl}
		~
		
		\begin{enumerate}[label=(\roman*),topsep=0pt]
			\item Let $C$ be a Hamiltonian cycle in some $\STS(v)$. Then switching $C$ results in a new $\STS(v)$ of the same isomorphism class as the original.
			\item Let $C_1$ and $C_2$ be complementary cycles in some $\STS(v)$. Then switching $C_1$ results in a new $\STS(v)$ of the same isomorphism class as switching $C_2$.
		\end{enumerate}
	\end{lemma}
	
	\begin{remark}
		As noted above, the smallest possible cycle length is 4, commonly referred to as a Pasch configuration. The blocks of a Pasch configuration have the form $\{a,b,c\}$, $\{a,d,e\}$, $\{b,e,f\}$ and $\{c,d,f\}$. This configuration may be viewed as the blocks arising from a 4-cycle in the cycle graph $C_{af}$, but also in the graphs $C_{bd}$ and $C_{ce}$. Thus every 4-cycle in an $\STS(v)$ arises from the cycle graphs of exactly three pairs of points. This situation does not occur for any other cycle lengths.
	\end{remark}
	
	Our principal interest in this paper is in equivalence classes of Steiner triple systems under cycle switching. We seek to determine, for a given $\STS(v)$, which isomorphism classes can be reached from the given system by a finite sequence of cycle switches of a given type. It is immediate that this defines an equivalence relation on the set of isomorphism classes of $\STS(v)$s, and so partitions the set. Alternatively, we may regard the set of $\STS(v)$s as the vertices of a graph, with a (directed) edge from $(V,\mathcal{B})$ to $(V,\mathcal{B}')$ whenever the isomorphism class of the latter system can be reached from the former by a cycle switch. Then the equivalence classes correspond to connected components in this graph; we shall use these two views of cycle switching equivalence interchangeably in what follows.
	
	More formally, the graph defined above is a multidigraph, with one directed edge emanating from a vertex for every cycle of the appropriate type contained in the corresponding $\STS(v)$. It is of course possible for a cycle switch to reach the same isomorphism class as the original system, and for more than one cycle to reach the same isomorphism class. Thus, loops and multiple directed edges are possible in the multidigraph. Of course, if there is a directed edge from $(V,\mathcal{B})$ to $(V,\mathcal{B}')$, then there must also be a directed edge in the opposite direction by considering the reverse switch. As our principal interest here is in connectivity, we may suppress the multiple edges and loops, and consider a pair of opposite directed edges as a single undirected edge, yielding a simple undirected graph. Allowing switching of any possible cycle length, we will denote the multidigraph of $\STS(v)$s by $\DSw(v)$, and the corresponding undirected graph by $\Sw(v)$. The above discussion shows that the connected components of $\Sw(v)$ and $\DSw(v)$ are the same, so we will work with whichever is most convenient. 
	
	Typically though, we will wish to restrict the allowable cycle switches to those of a particular length, or set of lengths. The cycle switching graph of $\STS(v)$s allowing only switching of cycles of length $\ell$ will be denoted by $\Sw[\ell](v)$, and similarly for the directed version. If we allow a restricted set of lengths $C\subset\{4,6,\ldots,v-7,v-3\}$ then the switching graph will be denoted by $\Sw[C](v)$.
	
	In~\cite{GGM}, the switching graphs for all $v\leq 15$ and all allowable cycle lengths were determined. Trivially, the graphs $\Sw[4](7)$ and $\Sw[6](9)$ both consist of a single vertex (with a suppressed loop). The graphs $\Sw[4](13)$ and $\Sw[6](13)$ both consist of two vertices representing the cyclic and non-cyclic systems joined by an edge. A suppressed loop on the vertex representing the non-cyclic system indicates that this system can also be switched to itself. This is not the case for the cyclic system. The graph $\Sw[10](13)$ consists of two isolated vertices (again, each with a suppressed loop).
	
	Referencing the standard numbering of the 80 $\STS(15)$s as given for example in~\cite{MPR}, the graph $\Sw[4](15)$ consists of one connected component consisting of 79 vertices, and an isolated vertex corresponding to the unique anti-Pasch system of this order (\#80). The graph $\Sw[6](15)$ consists of two connected components, one with 57 vertices and the other with 7 vertices (systems \#11, \#12, \#19, \#20, \#21, \#22 and \#61), as well as 16 isolated vertices corresponding to systems with no 6-cycle. The seven systems above all contain an $\STS(7)$ subsystem. The graph $\Sw[8](15)$ consists of one connected component containing 77 vertices, and three isolated vertices corresponding to systems \#1, \#16 and \#80 which have no 8-cycle. The graph $\Sw[12](15)$ consists of 80 isolated vertices, 78 of which have a suppressed loop; the other two correspond to systems \#1 and \#2 which have no 12-cycle.
	
	In~\cite{KMO}, it was shown that the switching graph $\Sw(19)$ of $\STS(19)$s is connected. In Section~\ref{sec:res}, we extend this result by showing that $\Sw[\ell](19)$ is disconnected for all $\ell\in\{4,6,8,10,12,16\}$, and determining all the connected components of these graphs. To our mind the results are a little surprising. particularly for 4-cycle switching. It was already known that the graph $\Sw[4](19)$ is disconnected and that there exist components consisting of just two vertices~\cite{GGM2}. Our expectation was that there would be one component containing nearly all systems and a number of other very small components each containing possibly only four or fewer systems. So it was interesting to find that there also exists a component containing 16 systems with an automorphism group of order 256, yet all systems have only the identity automorphism! The graph $\Sw[6](19)$ also contains a component of order 24. We also show that $\Sw[\{4,6\}](19)$ is connected, so the minimum set of cycle lengths required to ensure the connectivity of the switching graph has cardinality 2. The analysis done in this paper also posed computational challenges due to the large number of isomorphism classes of $\STS(19)$s and our algorithm to deal with these is described in Section~\ref{sec:alg} 
	
	When providing listings of particular systems, we will often employ the following compact notation which is common in the literature, see for example~\cite{FGG1}. The set $V$ of points is taken to be $\{0,1,2,\ldots,18\}$. The points $10,\ldots,18$ are represented by the letters $a,\ldots,i$ respectively. The 57 blocks of an $\STS(19)$ are represented by a string of symbols $s_1 s_2 \ldots s_{57}$. Using the usual lexicographical order, the symbol $s_i$ is the largest element $z_i$ in the $i$th triple $\{x_i,y_i,z_i\}$, where $x_i<y_i<z_i$. The remaining two elements implicitly have the property that there is no pair $x'_i<y'_i$ such that $\{x'_i,y'_i\}$ does not appear in an earlier triple, and either (i) $x'_i<x_i$ or (ii) $x'_i=x_i$ and $y'_i<y_i$.
	
	\section{Algorithm}\label{sec:alg}
	The principal practical problem to be considered when determining connected components of switching graphs of $\STS(19)$s is the large number of systems to be considered. There are 11,084,874,829 $\STS(19)$s~\cite{KO}, and so creating a full graph of this order in computer memory is entirely impractical. A further efficiency issue arises due to the fact that the vertices in the graph are \emph{isomorphism classes} of $\STS(19)$s. To compute the isomorphism class of a given system, we use the \texttt{nauty} package~\cite{nauty} to produce a canonical labelling of the point-block incidence graph. If carried out for every $\STS(19)$, this procedure would represent by far the most computationally expensive part of the process. We now describe an algorithm which avoids the necessity of canonically labelling the majority of the systems.
	
	Let $c$ be an allowable cycle length in an $\STS(19)$. Let $u,v$ be $\STS(19)$s, considered as vertices in the switching graph $\DSw[c](19)$. We write $u\to v$ if there is a directed edge from $u$ to $v$ in the graph; that is to say, there is some $c$-cycle in $u$ which reaches the isomorphism class of $v$ when switched. Let $N_c(u)$ denote the number of $c$-cycles in the vertex $u$. We define a strict partial order $<$ on the vertices of $\DSw[c](19)$ as follows:
	\[
	u<v\text{ iff } u\to v \text{ and } N_c(u)<N_c(v).
	\]
	
	In other words, $u<v$ if $v$ is reached by a $c$-cycle switch from $u$, \emph{and} $v$ contains strictly more $c$-cycles than $u$. Note that by definition, every connected component in $\DSw[c](19)$, and hence also in $\Sw[c](19)$, contains a vertex which is a maximal element in this preorder. Thus to compute the connected components of the graph, it suffices to consider only these maximal elements.
	
	\textbf{Step 1.} For each $\STS(19)$ $u$ we find all the $c$-cycles in the system, and for each cycle we find the system $v$ reached by the switch. If $N_c(v)>N_c(u)$ then we discard $u$.
	
	\textbf{Step 2.} For each system remaining after Step 1, we extend the logic by discarding a system $u$ if there is some system $v$ at distance at most $k$ from $u$ in $\DSw[c](v)$ with $N_c(v)>N_c(u)$. This is done by recursively exploring the switched systems starting from $u$, stopping to discard $u$ whenever a system with more $c$-cycles is encountered. This becomes an expensive operation, but typically a value of $k=4$ is sufficient to reduce the remaining population of vertices to be considered to manageable levels.
	
	\textbf{Step 3.} The remaining systems after Step 2 will often comprise a number of small connected components in the switching graph. To detect this, we select a vertex $u$ in the set and recursively explore the switching graph until a limit $N$ is reached, or no new vertices are detected. Typically, a value of $N$ around 100 is sufficient to find most small components.
	
	\textbf{Step 4.} Systems remaining after Step 3 will fall into one or more larger connected components. Let $V=\{v_1,v_2,\ldots,v_m\}$ be the set of remaining vertices at this stage. By construction, every vertex of $V$ is connected to at least one vertex of $\Sw[c](19)$. Now choose for each $i=1,\ldots,m$ a positive integer $k_i$ and construct the closed $k_i$-neighbourhood $\overline{N}_{k_i}(v_i)$, i.e. the set of vertices in $\Sw[c](19)$ at distance no more than $k_i$ from $v_i$.  Then define a graph $W$ with vertex set $V$ and an edge from $v_i$ to $v_j$ if $\overline{N}_{k_i}(v_i)\cap\overline{N}_{k_j}(v_j)\neq\emptyset$. Then if $W$ is connected, it follows that $\Sw[c](19)$ is connected.
	
	\textbf{Step 5.} If the graph $W$ from Step 4 is not connected, then each connected component of $W$ is a subset of some connected component of $\Sw[c](19)$, but two connected components of $W$ may in fact be in the same connected component of $\Sw[c](19)$. To deal with this possibility, we select a representative vertex from each connected component of $W$ and repeat Step 3 with a much larger value of $N$ to detect any large connected components. In practice, a value of $N=1,000,000$ proved sufficient to detect all such components.
	
	\section{Results}\label{sec:res}
	Before describing our detailed results, we begin with a small lemma which relates the connected components of $\Sw[10](19)$ to those of $\Sw[6](19)$, and the components of $\Sw[12](19)$ to those of $\Sw[4](19)$. 
	
	\begin{lemma}\label{lem:subg}
		Let $v\geq 13$ and let $u$ be a vertex in $\Sw(v)$ representing a given $\STS(v)$. Then the connected component in $\Sw[v-7](v)$ containing $u$ is a subgraph of the connected component of $\Sw[4](v)$ containing $u$; and the connected component in $\Sw[v-9](v)$ containing $u$ is a subgraph of the connected component of $\Sw[6](v)$ containing $u$.
	\end{lemma}
	\begin{proof}
		Every $(v-7)$-cycle is necessarily complementary to a 4-cycle, and so by Lemma~\ref{lem:compl}, any isomorphism class reached by a $(v-7)$-cycle switch is also reached by some 4-cycle switch. A similar argument holds for $(v-9)$-cycles and 6-cycles.
	\end{proof}
	
	We now summarise the results for the cycle switching graphs $\Sw[\ell](19)$, where $\ell\in\{4,6,8,10,12\}$. Note that by Lemma~\ref{lem:compl}(i), the graph $\DSw[16](19)$ consists of a set of isolated vertices, possibly with a  number of loops at each vertex. Thus the connected components of $\Sw[16](19)$ are 11,084,874,829 isolated vertices.
	
	\subsection{Switching 4-cycles}
	We carried out the algorithm described in Section~\ref{sec:alg}, using a maximum distance of 4 in step 2. After steps 2 and 3, there were 2,591 identified systems without a 4-cycle. An $\STS$ without 4-cycles is also known as an \emph{anti-Pasch} system; the 2,591 anti-Pasch systems agrees with the results in~\cite{EIGHT}. Step 3 also identified 139 components containing between two and four systems each; these are illustrated in Figure~\ref{fig:small} as types (b) to (f). Component type (b) consists of two $\STS(19)$s with a single 4-cycle each, with the property that switching the 4-cycle in one system reaches the isomorphism class of the other. Such an arrangement is known as a \emph{twin} system, a concept first introduced in~\cite{GGM2}. The count of 126 pairs of twins agrees with~\cite{EIGHT}.
	
	\begin{figure}\centering
		\begin{tabular}{M{1cm}M{7cm}M{1.5cm}}
			\multicolumn{2}{c}{Component type} & Number \\
			\hline\hline
			& & \\
			
			(a) &
			\begin{tikzpicture}[x=0.3mm,y=0.3mm,thick,vertex/.style={circle,draw,minimum size=10,inner sep=0pt,fill=lightgray}]
				\node at (0,0) [vertex] (v1) {};
			\end{tikzpicture}
			& 2,591 \\
			
			(b) &
			\begin{tikzpicture}[x=0.3mm,y=0.3mm,thick,vertex/.style={circle,draw,minimum size=10,inner sep=0pt,fill=lightgray}]
				\node at (-40,0) [vertex] (v1) {};
				\node at (40,0) [vertex] (v2) {};
				\draw [<->] (v1) to (v2);
			\end{tikzpicture}
			& 126 \\
			
			(c) &
			\begin{tikzpicture}[x=0.3mm,y=0.3mm,thick,vertex/.style={circle,draw,minimum size=10,inner sep=0pt,fill=lightgray}]
				\node at (-40,0) [vertex] (v1) {};
				\node at (40,0) [vertex] (v2) {};
				\draw [<->] (v1) to (v2);
				\draw [->,loop left, min distance=1.5cm,in=225,out=135] (v1) to (v1);
			\end{tikzpicture}
			& 1 \\
			
			(d) &
			\begin{tikzpicture}[x=0.3mm,y=0.3mm,thick,vertex/.style={circle,draw,minimum size=10,inner sep=0pt,fill=lightgray}]
				\node at (-30,0) [vertex] (v1) {};
				\node at (30,0) [vertex] (v2) {};
				\node at (90,0) [vertex] (v3) {};
				\draw [<->] (v1) to (v2);
				\draw [<->] (v3) to (v2);
			\end{tikzpicture}
			& 9 \\
			
			(e) &
			\begin{tikzpicture}[x=0.3mm,y=0.3mm,thick,vertex/.style={circle,draw,minimum size=10,inner sep=0pt,fill=lightgray}]
				\node at (-20,0) [vertex] (v1) {};
				\node at (20,0) [vertex] (v2) {};
				\node at (60,0) [vertex] (v3) {};
				\node at (100,0) [vertex] (v4) {};
				\draw [<->] (v1) to (v2);
				\draw [<->] (v3) to (v2);
				\draw [<->] (v3) to (v4);
			\end{tikzpicture}
			& 1 \\
			
			(f) &
			\begin{tikzpicture}[x=0.3mm,y=0.3mm,thick,vertex/.style={circle,draw,minimum size=10,inner sep=0pt,fill=lightgray}]
				\node at (-40,0) [vertex] (v1) {};
				\node at (40,0) [vertex] (v2) {};
				\node at (0,-20) [vertex] (v3) {};
				\node at (0,20) [vertex] (v4) {};
				\draw [<->] (v1) to (v4);
				\draw [<->] (v2) to (v4);
				\draw [<->] (v1) to (v3);
				\draw [<->] (v2) to (v3);
			\end{tikzpicture}
			& 2 \\
		\end{tabular}
		
		\begin{tabular}{cl}
			Type & Example representative system \\
			\hline
			(a) & \texttt{2468acegiad7ghbei5f89gdih9hbdfibhaiegfiehgcfeigfidcgfhihh} \\
			(b) & \texttt{2468acegih7cdagifga9hdifef7eibdbe9ghiahigbfciegihfgdhfihh} \\
			(c) & \texttt{2468acegiefcdhaig578fcdhi9ibfghhg9deiagdifebhicfgeihhgihf} \\
			(d) & \texttt{2468acegihfceabig9c7adfih8gefiddihgbefheibhdfgifecihgihgh} \\
			(e) & \texttt{2468acegi85fhideg9ibdacghgachif8dgbfh9hfdiebihgifeigehhif} \\
			(f) & \texttt{2468acegiceh7gfbii7adbcfhg8aefhchfigde9ifbifghgdidhigehhi} \\
		\end{tabular}
		
		\caption{Small connected components in $\DSw[4](19)$}
		\label{fig:small}
	\end{figure}
	
	In addition to these small components, step 3 identified one more component containing 16 systems, each having exactly six 4-cycles. These 4-cycles form a structure consisting of six Pasch configurations collectively containing 18 points and 24 blocks. Thus the Pasch configurations are block disjoint. The structures fall into two isomorphism classes ($A$ and $B$), each having automorphism group $C_2\times C_2\times C_2\times D_8$ of order 64. This totally unexpected and somewhat mysterious component is illustrated in Figure~\ref{fig:comp16}, with the labelling of the vertices indicating the isomorphism class of the structure containing the six Pasch configurations of the $\STS(19)$. The component is bipartite, but does not respect the isomorphism partitioning of the Pasch configuration structures. Its automorphism group has order 256 and acts with two orbits, one on the eight vertices incident with a double directed edge and the other on the remaining eight vertices. In contrast to the orders of the automorphism groups of the component and of the Pasch configuration structures, every $\STS(19)$ in the component has trivial automorphism group!
	
	\begin{figure}\centering
		\input{comp16}
		
		\begin{tabular}{ll}
			& \\
			$A0$ & \texttt{2468acegi6d8bhcfid5hiafge7cibghiahgefdfighgecdfehfbgiihhi} \\
			$A1$ & \texttt{2468acegid7ahfbei75gbehfi9igcehfcieghdhigfe9bdfghidihgifh} \\
			$A2$ & \texttt{2468acegi6d8bhcfid5iahgef7cibghhiafgedfighgecdfehfbgiihhi} \\
			$A3$ & \texttt{2468acegid7ahfbei75fciegh9igcehgbehfidhigfe9bdfghidihgifh} \\
			$A4$ & \texttt{2468acegicha89dfi9gbieahffhabig7f9eiddighgdcefihfihgehhgi} \\
			$A5$ & \texttt{2468acegic8f9behii9gefchdd7ghfeehdcgiihcbbigfeghfdiighifh} \\
			$A6$ & \texttt{2468acegifcdbeahiheagibdf7c9digb8aighfgihhfeibgfdiheghihf} \\
			$A7$ & \texttt{2468acegi97cdfehidi7hagfebecgih98feghhgifficgehigbdihfdih} \\
			$B0$ & \texttt{2468acegigbd79ihfi8ehdfcg9cfaehad9ighihfggebighedfiidhfhi} \\
			$B1$ & \texttt{2468acegiad97fheie5ic9hfgdbigchhfegcighiefafdgbdhdihigfhi} \\
			$B2$ & \texttt{2468acegie8b7fhdiidchbfegfadgchgcfibhihdeegfihcgbdiigifhh} \\
			$B3$ & \texttt{2468acegie8c9dfhi59fabghi79gdihbihefgefghidbhgcgiiehifhdf} \\
			$B4$ & \texttt{2468acegi8cf7hedii7gahdfe9bhdgfahebgiedihifdggbficfgiihhh} \\
			$B5$ & \texttt{2468acegic8f9behiibgefcdhd7ghfeehdgciihcbbigfeghfdiighifh} \\
			$B6$ & \texttt{2468acegie57fdchiagbfc9hic9igdhh9deifghefibdigbfhieghihfg} \\
			$B7$ & \texttt{2468acegiec7bdfhiid8hbegffah9dgg8ahfideihfcigiegbifhgdihh} \\
\end{tabular}
		\caption{The connected component of order 16 in $\DSw[4](19)$}
		\label{fig:comp16}
	\end{figure}
	
	At step 4, there were 83 systems remaining. As a first attempt we chose $k_i=4$ for all $i$. The corresponding graph $W$ turned out not to be connected; three vertices were isolated. However, choosing larger values of $k_i$ for these awkward vertices, we were able to show that the graph $W$ in step 4 is connected. Thus all the remaining $\STS(19)$s form a single connected component, and step 5 was not required.
	
	\newpage
	Summarising these results we have the following theorem.
	
	\begin{theorem}\label{thm:c4}
		The Pasch (4-cycle) switch graph $\Sw[4](19)$ consists of 2,732 connected components as follows.
		\begin{itemize}[itemsep=-1ex]
			\item 2,591 anti-Pasch systems as in Figure~\ref{fig:small}(a);
			\item 126 pairs of twin systems as in Figure~\ref{fig:small}(b);
			\item one component as in Figure~\ref{fig:small}(c);
			\item nine components as in Figure~\ref{fig:small}(d);
			\item one component as in Figure~\ref{fig:small}(e);
			\item two components as in Figure~\ref{fig:small}(f);
			\item one component as in Figure~\ref{fig:comp16};
			\item one component containing the remaining 11,084,871,929 systems.
		\end{itemize}
	\end{theorem}
	
	\subsection{Switching 6-cycles}
	As in~\cite{EIGHT}, we find a single 6-cycle free system. On base set $\{\infty,A_i,B_i,C_i:i\in\mathbb{Z}_6\}$ the blocks are the orbits generated by the starters $\{\infty,A_0,A_3\}$, $\{\infty,B_0,B_3\}$, $\{\infty,C_0,C_3\}$, $\{A_0,A_1,B_0\}$, $\{A_0,A_2,C_1\}$, $\{B_0,B_1,C_1\}$, $\{B_0,B_2,A_4\}$, $\{C_0,C_1,B_2\}$, $\{C_0,C_2,A_4\}$, $\{A_0,B_3,C_0\}$, $\{A_0,B_1,C_3\}$ under the permutation $(\infty)(A_0\ A_1\ldots A_5)(B_0\ B_1\ldots B_5)(C_0\ C_1\ldots C_5)$. The system is therefore 3-rotational.
	
	
	There is one component of order 24, consisting of systems each containing 36 6-cycles. This component is illustrated in Figure~\ref{fig:c6comp24}; because of the large out-degree of each vertex, we have suppressed multiple edges and loops, so the figure represents the undirected component of $\Sw[6](19)$. Note that as an undirected graph, this component is not regular; the vertex degrees range from 9 to 19. The automorphism group of the undirected component has order 4. A representative system, from which the component may be reconstructed by 6-cycle switching, is the following. 
	
	\begin{config}
		2468acegif58aehdi5ha9gdficeidhgegfbdic9ieghhgiffibehhigfh
	\end{config}
	
	\begin{figure}
		\centering
		\input{c6comp}
		\caption{The connected component of order 24 in $\Sw[6](19)$}
		\label{fig:c6comp24}
	\end{figure}
	
	This time, we were unable to show that the graph $W$ in step 4 is connected, but construction of the graph revealed that the remaining systems formed at most two connected components. Thus we used step 5 and identified a further large connected component containing 284,433 vertices. We summarise these results in the following theorem.
	
	\begin{theorem}\label{thm:c6}
		The four connected components of $\Sw[6](19)$ are as follows.
		\begin{itemize}[itemsep=-1ex]
			\item one 6-cycle free system;
			\item one component of order 24 as in Figure~\ref{fig:c6comp24};
			\item one component containing 284,433 systems;
			\item one component containing the remaining 11,084,590,371 systems.
		\end{itemize}
	\end{theorem}
	
	It is noteworthy that representatives of the four connected components in $\Sw[6](19)$ are connected via 4-cycle switching. Since by~\cite{EIGHT} there exists an $\STS(19)$ without a $c$-cycle for all $c\in\{4,6,8,10,12,16\}$, it follows that every graph $\Sw[c](19)$ is disconnected. We thus have the following theorem.
	
	\begin{theorem}
		The minimum cardinality of a set $C$ such that $\Sw[C](19)$ is connected is 2, and an example set is $C=\{4,6\}$.
	\end{theorem}
	
	\subsection{Switching 8-cycles}
	By~\cite{EIGHT}, there are exactly 381 8-cycle free systems and step 3 identified these; therefore there are 381 isolated vertices in $\Sw[8](19)$ of this form. There is one further isolated vertex. On base set $V=\{X,Y,Z,A_i,B_i:i\in\mathbb{Z}_8\}$, the blocks are $\{X,Y,Z\}$ together with the orbits generated by the starters $\{X,A_0,A_1\}$, $\{Y,B_0,B_1\}$, $\{Z,A_0,A_4\}$, $\{Z,B_0,B_4\}$, $\{A_0,A_2,B_3\}$, $\{A_0,A_3,B_0\}$, $\{B_0,B_2,A_4\}$, $\{B_0,B_3,A_1\}$ under the permutation $(Z)(X\ Y)(A_0\ A_1\ldots A_7)(B_0\ B_1\ldots B_7)$. The cycle graph on the pair $\{X,Y\}$ has two 8-cycles in a complementary pair. By Lemma~\ref{lem:compl}, these both switch to the same isomorphism class, which in this case is the original system. There are no other 8-cycles. The cycle graphs on the pairs $\{X,Z\}$ and $\{Y,Z\}$ have four 4-cycles; on the pairs $\{A_i,B_i\}$, $i\in\mathbb{Z}_8$ a 6-cycle and a 10-cycle; and on the pairs $\{A_i,A_{i+3}\}$ and $\{B_i,B_{i+3}\}$, $i\in\mathbb{Z}_8$ a 4-cycle and a 12-cycle. On all other pairs the cycle graph is a Hamiltonian cycle.
	
	%
	
	At step 4, we were able to show that all other vertices in $\Sw[8](19)$ are in a single connected component. We summarise these results in the following theorem.
	\begin{theorem}\label{thm:c8}
		The 383 connected components of $\Sw[8](19)$ are as follows.
		\begin{itemize}[itemsep=-1ex]
			\item 381 8-cycle free systems;
			\item one further system with two complementary 8-cycles,switching to the original system; 
			\item one component containing the remaining 11,084,874,447 systems.
		\end{itemize}
	\end{theorem}
	
	\subsection{Switching 10-cycles}
	We find 66 10-cycle free systems, in agreement with~\cite{EIGHT}. By Lemma~\ref{lem:compl}, every connected component of $\Sw[10](19)$ is a subset of some connected component of $\Sw[6](19)$. However, there may be 6-cycles in a given $\STS(19)$ which are not complementary to a 10-cycle, and so the components of $\Sw[6](19)$ may be disconnected in $\Sw[10](19)$. 
	
	The unique 6-cycle free system is of course also 10-cycle free. There is also a connected component of order 24 consisting of the same systems as the component of the same order in the 6-cycle switching graph. Interestingly, and as we have already observed, although the graph of this component in the 6-cycle switching graph is not regular (Figure~\ref{fig:c6comp24}), it is regular of degree 36 in the 10-cycle switching graph!
	
	The connected component of cardinality 284,433 in $\Sw[6](19)$ fractures into 76 components, 65 of which are isolated vertices corresponding to 10-cycle free systems. There are 10 components of ``moderate'' size which are summarised in Table~\ref{tab:c10}. In the table, by ``degree'' we mean the out-degree of a vertex in $\DSw[10](19)$; in other words, the number of 10-cycles in each system in the component. It is an unexpected feature of these components that, with the single exception of the largest component of order 1,078, all systems have the same number of 10-cycles. The remaining component has order 282,123.
	
	The largest connected component in $\Sw[6](19)$ fractures into just two components. One of these has order 56 and is regular of degree 18. A representative system is the following.
	
	\begin{config}
		2468acegib6cfgdih5cbhegdicgdhfibi9hgf9feihahifgedcifeghhi
	\end{config}
	
	The other component contains the remaining 11,084,590,315 systems.
	
	\begin{table}
		\centering\small
		\begin{tabular}{|l|l|l|l|}
			\hline
			\# & Order & Degree & Representative system\\
			\hline
			1 & 84 & 18 & 
			\texttt{2468acegi687hcigf758eghdi8cfghi7difhgideghbeifhfhdeigbfih} \\
			2 & 152 & 30 & 
			\texttt{2468acegi6a9cedih95ahcgfiaebdhi9gdfihifeghfhiegdbighfghif} \\
			3 & 152 & 24 & 
			\texttt{2468acegi678ebdih857hcfig7cedhi8fgeihbfhgidhigfgihefidgfh} \\
			4 & 152 & 18 & 
			\texttt{2468acegi6bcg9fhic5bedhgibdhfigchgidffedhiafgihbicieefhgh} \\
			5 & 294 & 30 & 
			\texttt{2468acegii6geacdhahf89idghecgdfafbegiidbeafhigghifhhehfii} \\
			6 & 1078 & 26 (432) & 
			\texttt{2468acegia9bfhegie7gcdhif89fidhedgbhidihfiaghgcefhidegfhi} \\
			&  & 28 (432) & \\
			&  & 32 (214) & \\
			7 & 108 & 20 & 
			\texttt{2468acegif79cigehg9dfbhiehbicdeeabdihfagiheigedggfhffhihi} \\
			8 & 108 & 28 & 
			\texttt{2468acegifecgahidcdgbi9hf8iebdhh8bgcifebicdfhhgigeffgihhi} \\
			9 & 98 & 27 & 
			\texttt{2468acegid7h9bgif7afhibdgig9fehbihefggedcfadecbhigfihihhi} \\
			10 & 19 & 27 & 
			\texttt{2468acegi578ibfgh867edfhi7hcgif8gihefbfdihdeighcgehifhdig} \\
			\hline
		\end{tabular}
		\caption{Components in $\DSw[10](19)$}
		\label{tab:c10}
	\end{table}
	
	\newpage
	We summarise the results for 10-cycle switching below.
	
	\begin{theorem}\label{thm:c10}
		The 80 connected components of $\Sw[10](19)$ are as follows.
		\begin{itemize}[itemsep=-1ex]
			\item 66 10-cycle free systems;
			\item one component containing 24 systems, being the same systems as those of the component of cardinality 24 in $\Sw[6](19)$;
			\item 10 components as in Table~\ref{tab:c10} containing a total of 2,245 systems;
			\item one component containing 282,123 systems;
			\item one component containing 56 systems;
			\item one component containing 11,084,590,315 systems.
		\end{itemize}
	\end{theorem}
	
	\subsection{Switching 12-cycles}
	By Lemma~\ref{lem:subg}, every component of $\Sw[12](19)$ is a subset of some component of $\Sw[4](19)$. There are 2,727 systems with no 12-cycle, again in agreement with~\cite{EIGHT}. There are a further 147 small components identified at step 3 of the algorithm, as pictured in Figure~\ref{fig:c12}. 
	
	Component types (c) and (d) are interesting. There are a total of 126 such components; the systems involved are precisely the 126 pairs of twins which appear as components in the 4-cycle switching graph. Note that although the systems in these components have two or three 12-cycles, each has only a single 4-cycle; this situation is explained by the remark in Section~\ref{sec:intro}, since a 4-cycle may be complementary to up to three 12-cycles.
	
	Components (k) and (l) are subsets of the component containing 16 systems in $\DSw[4](19)$. In Figure~\ref{fig:c12}, the vertices in these components are identified with the corresponding labels from Figure~\ref{fig:comp16}.
	
	At step 4, the remaining systems all form a single component. We summarise these results in the following theorem.
	
	\begin{theorem}\label{thm:c12}
		The 2,825 connected components of $\Sw[12](19)$ are as follows.
		\begin{itemize}[itemsep=-1ex]
			\item 2,727 12-cycle free systems;
			\item 147 components as in Figure~\ref{fig:c12} (b) to (l) containing a total of 320 systems;
			\item one component containing the remaining 11,084,871,782 systems.
		\end{itemize}
	\end{theorem}
	
	\begin{figure}\centering\small
		\begin{tabular}{M{1cm}M{6cm}M{1.5cm}}
			\multicolumn{2}{c}{Component type} & Number \\
			\hline\hline
			& & \\
			
			(a) &
			\begin{tikzpicture}[x=0.3mm,y=0.3mm,thick,vertex/.style={circle,draw,minimum size=10,inner sep=0pt,fill=lightgray}]
				\node at (0,0) [vertex] (v1) {};
			\end{tikzpicture}
			& 2,727 \\
			
			(b) &
			\begin{tikzpicture}[x=0.3mm,y=0.3mm,thick,vertex/.style={circle,draw,minimum size=10,inner sep=0pt,fill=lightgray}]
				\node at (-40,0) [vertex] (v1) {};
				\node at (40,0) [vertex] (v2) {};
				\draw [<->] (v1) to (v2);
			\end{tikzpicture}
			& 6 \\
			
			(c) &
			\begin{tikzpicture}[x=0.3mm,y=0.3mm,thick,vertex/.style={circle,draw,minimum size=10,inner sep=0pt,fill=lightgray}]
				\node at (-40,0) [vertex] (v1) {};
				\node at (40,0) [vertex] (v2) {};
				\draw [<->,bend left=20] (v1) to (v2);
				\draw [<->,bend right=20] (v1) to (v2);
			\end{tikzpicture}
			& 42 \\
			
			(d) &
			\begin{tikzpicture}[x=0.3mm,y=0.3mm,thick,vertex/.style={circle,draw,minimum size=10,inner sep=0pt,fill=lightgray}]
				\node at (-40,0) [vertex] (v1) {};
				\node at (40,0) [vertex] (v2) {};
				\draw [<->] (v1) to (v2);
				\draw [<->,bend left=20] (v1) to (v2);
				\draw [<->,bend right=20] (v1) to (v2);
			\end{tikzpicture}
			& 84 \\
			
			(e) &
			\begin{tikzpicture}[x=0.3mm,y=0.3mm,thick,vertex/.style={circle,draw,minimum size=10,inner sep=0pt,fill=lightgray}]
				\node at (-40,0) [vertex] (v1) {};
				\node at (40,0) [vertex] (v2) {};
				\draw [<->] (v1) to (v2);
				\draw [<->,bend left=20] (v1) to (v2);
				\draw [<->,bend right=20] (v1) to (v2);
				\draw [->,loop above, min distance=0.9cm,in=125,out=55] (v1) to (v1);
				\draw [->,loop left, min distance=0.9cm,in=215,out=145] (v1) to (v1);
				\draw [->,loop below, min distance=0.9cm,in=305,out=235] (v1) to (v1);
			\end{tikzpicture}
			& 1 \\
			
			(f) &
			\begin{tikzpicture}[x=0.3mm,y=0.3mm,thick,vertex/.style={circle,draw,minimum size=10,inner sep=0pt,fill=lightgray}]
				\node at (-30,0) [vertex] (v1) {};
				\node at (30,0) [vertex] (v2) {};
				\node at (90,0) [vertex] (v3) {};
				\draw [<->,bend left=20] (v1) to (v2);
				\draw [<->,bend right=20] (v1) to (v2);
				\draw [<->,bend left=20] (v2) to (v3);
				\draw [<->,bend right=20] (v2) to (v3);
			\end{tikzpicture}
			& 1 \\
			
			(g) &
			\begin{tikzpicture}[x=0.3mm,y=0.3mm,thick,vertex/.style={circle,draw,minimum size=10,inner sep=0pt,fill=lightgray}]
				\node at (-30,0) [vertex] (v1) {};
				\node at (30,0) [vertex] (v2) {};
				\node at (90,0) [vertex] (v3) {};
				\draw [<->,bend left=20] (v1) to (v2);
				\draw [<->] (v1) to (v2);
				\draw [<->,bend right=20] (v1) to (v2);
				\draw [<->,bend left=20] (v2) to (v3);
				\draw [<->,bend right=20] (v2) to (v3);
			\end{tikzpicture}
			& 2 \\
			
			(h) &
			\begin{tikzpicture}[x=0.3mm,y=0.3mm,thick,vertex/.style={circle,draw,minimum size=10,inner sep=0pt,fill=lightgray}]
				\node at (-30,0) [vertex] (v1) {};
				\node at (30,0) [vertex] (v2) {};
				\node at (90,0) [vertex] (v3) {};
				\draw [<->,bend left=20] (v1) to (v2);
				\draw [<->] (v1) to (v2);
				\draw [<->,bend right=20] (v1) to (v2);
				\draw [<->,bend left=20] (v2) to (v3);
				\draw [<->] (v2) to (v3);
				\draw [<->,bend right=20] (v2) to (v3);
			\end{tikzpicture}
			& 7 \\
			
			(i) &
			\begin{tikzpicture}[x=0.3mm,y=0.3mm,thick,vertex/.style={circle,draw,minimum size=10,inner sep=0pt,fill=lightgray}]
				\node at (-40,0) [vertex] (v1) {};
				\node at (40,0) [vertex] (v2) {};
				\node at (0,-20) [vertex] (v3) {};
				\node at (0,20) [vertex] (v4) {};
				\draw [<->,bend left=20] (v1) to (v4);
				\draw [<->,bend left=20] (v2) to (v4);
				\draw [<->,bend left=20] (v1) to (v3);
				\draw [<->,bend left=20] (v2) to (v3);
				\draw [<->,bend left=20] (v4) to (v1);
				\draw [<->,bend left=20] (v4) to (v2);
				\draw [<->,bend left=20] (v3) to (v1);
				\draw [<->,bend left=20] (v3) to (v2);
			\end{tikzpicture}
			& 1 \\
			
			(j) &
			\begin{tikzpicture}[x=0.3mm,y=0.3mm,thick,vertex/.style={circle,draw,minimum size=10,inner sep=0pt,fill=lightgray}]
				\node at (-40,0) [vertex] (v1) {};
				\node at (40,0) [vertex] (v2) {};
				\node at (0,-20) [vertex] (v3) {};
				\node at (0,20) [vertex] (v4) {};
				\draw [<->,bend left=20] (v1) to (v4);
				\draw [<->] (v1) to (v4);
				\draw [<->,bend left=20] (v2) to (v4);
				\draw [<->] (v2) to (v4);
				\draw [<->,bend left=20] (v1) to (v3);
				\draw [<->] (v1) to (v3);
				\draw [<->,bend left=20] (v2) to (v3);
				\draw [<->] (v2) to (v3);
				\draw [<->,bend left=20] (v4) to (v1);
				\draw [<->,bend left=20] (v4) to (v2);
				\draw [<->,bend left=20] (v3) to (v1);
				\draw [<->,bend left=20] (v3) to (v2);
			\end{tikzpicture}
			& 1 \\
			
			(k) & \input{c12comp8a} & 1 \\
			(l) & \input{c12comp8b} & 1 \\
		\end{tabular}
		
		\begin{tabular}{cl}
			Type & Example representative system \\
			\hline
			(a) & \texttt{2468acegibfegchdif6d9cehi7idhcgheicgdaigfabfdhighgfefhihi} \\
			(b) & \texttt{2468acegig5aebhfih7b8dcfiaec9ifidcfghdfiehhfdggbihigegihh} \\
			(c) & \texttt{2468acegic5igdhbf6dc9bghi8eifghfaicehfehigbhgdehifegdihfi} \\
			(d) & \texttt{2468acegifbcig9dha67iegdh8beghiicfdhghfeg9dcfhbhigheififi} \\
			(e) & \texttt{2468acegiefcdhaig578fcdhi9ibfghhg9deiagdifebhicfgeihhgihf} \\
			(f) & \texttt{2468acegi68eaigdh9gb7hcfifebdih7afheidchi9ifghgdhiigfehfg} \\
			(g) & \texttt{2468acegihfceabig9c7adfih8gefiddihgbefheibhdfgifecihgihgh} \\
			(h) & \texttt{2468acegi6dge9hif7gi9hcdfafehgi7ecbihdcfhibfdhgbhfeihiggi} \\
			(i) & \texttt{2468acegi59gafheid7cebfhi9hcfigdhgbfigihefcadiiefehgdighh} \\
			(j) & \texttt{2468acegiceh7gfbii7adbcfhg8aefhchfigde9ifbifghgdidhigehhi} \\
			(k) & \texttt{2468acegi6d8bhcfid5hiafge7cibghiahgefdfighgecdfehfbgiihhi} \\
			(l) & \texttt{2468acegicha89dfi9gbieahffhabig7f9eiddighgdcefihfihgehhgi} \\
		\end{tabular}
		\caption{Small connected components in $\DSw[12](19)$}
		\label{fig:c12}
	\end{figure}
	

\end{document}

%% file: comp16.tex
\begin{tikzpicture}[x=0.2mm,y=0.2mm,very thick,vertex/.style={circle,draw,minimum size=18,fill=white,inner sep=0pt}]
	\footnotesize
	\node at (170.4,-99.1) [vertex] (v1) {$A2$};
	\node at (95.6,-174.2) [vertex] (v2) {$B2$};
	\node at (-170.1,-99.9) [vertex] (v3) {$A0$};
	\node at (287.5,-123.3) [vertex] (v4) {$B6$};
	\node at (-287.7,115) [vertex] (v5) {$B7$};
	\node at (169.9,91.5) [vertex] (v6) {$A1$};
	\node at (-95,-174.7) [vertex] (v7) {$B0$};
	\node at (119,-291.8) [vertex] (v8) {$A6$};
	\node at (-119.2,283.5) [vertex] (v9) {$A7$};
	\node at (94.8,166.3) [vertex] (v10) {$B1$};
	\node at (287.5,115) [vertex] (v11) {$B5$};
	\node at (-287.8,-123.2) [vertex] (v12) {$B4$};
	\node at (-170.6,90.7) [vertex] (v13) {$A3$};
	\node at (119.1,283.4) [vertex] (v14) {$A5$};
	\node at (-119.3,-291.8) [vertex] (v15) {$A4$};
	\node at (-95.8,165.7) [vertex] (v16) {$B3$};
	\draw [<->,bend left=20] (v1) to (v2);
	\draw [<->,bend left=20] (v2) to (v1);
	\draw [<->] (v1) to (v3);
	\draw [<->] (v1) to (v4);
	\draw [<->] (v1) to (v5);
	\draw [<->] (v1) to (v6);
	\draw [<->] (v2) to (v7);
	\draw [<->] (v2) to (v8);
	\draw [<->] (v2) to (v9);
	\draw [<->] (v2) to (v10);
	\draw [<->,bend left=20] (v3) to (v7);
	\draw [<->,bend left=20] (v7) to (v3);
	\draw [<->] (v3) to (v11);
	\draw [<->] (v3) to (v12);
	\draw [<->] (v3) to (v13);
	\draw [<->] (v4) to (v8);
	\draw [<->] (v4) to (v9);
	\draw [<->] (v4) to (v11);
	\draw [<->] (v4) to (v12);
	\draw [<->] (v4) to (v13);
	\draw [<->] (v5) to (v8);
	\draw [<->] (v5) to (v9);
	\draw [<->] (v5) to (v11);
	\draw [<->] (v5) to (v12);
	\draw [<->] (v5) to (v13);
	\draw [<->,bend left=20] (v6) to (v10);
	\draw [<->,bend left=20] (v10) to (v6);
	\draw [<->] (v6) to (v11);
	\draw [<->] (v6) to (v12);
	\draw [<->] (v6) to (v13);
	\draw [<->] (v7) to (v14);
	\draw [<->] (v7) to (v15);
	\draw [<->] (v7) to (v16);
	\draw [<->] (v8) to (v14);
	\draw [<->] (v8) to (v15);
	\draw [<->] (v8) to (v16);
	\draw [<->] (v9) to (v14);
	\draw [<->] (v9) to (v15);
	\draw [<->] (v9) to (v16);
	\draw [<->] (v10) to (v14);
	\draw [<->] (v10) to (v15);
	\draw [<->] (v10) to (v16);
	\draw [<->] (v11) to (v14);
	\draw [<->] (v11) to (v15);
	\draw [<->] (v12) to (v14);
	\draw [<->] (v12) to (v15);
	\draw [<->,bend left=20] (v13) to (v16);
	\draw [<->,bend left=20] (v16) to (v13);
\end{tikzpicture}

%% file: c6comp.tex
\begin{tikzpicture}[x=0.22mm,y=0.22mm,thick,vertex/.style={circle,draw,minimum size=10,fill=lightgray}]
	\node at (-40.8,-263.4) [vertex] (v1) {};
	\node at (-108.2,-245.2) [vertex] (v2) {};
	\node at (-168.6,-210.3) [vertex] (v3) {};
	\node at (-217.9,-160.9) [vertex] (v4) {};
	\node at (241.1,-101) [vertex] (v5) {};
	\node at (-270.7,-33.1) [vertex] (v6) {};
	\node at (-270.6,36.7) [vertex] (v7) {};
	\node at (-252.5,104) [vertex] (v8) {};
	\node at (259.2,-33.6) [vertex] (v9) {};
	\node at (-217.5,164.4) [vertex] (v10) {};
	\node at (241.3,103.6) [vertex] (v11) {};
	\node at (206.2,-161.3) [vertex] (v12) {};
	\node at (156.8,-210.6) [vertex] (v13) {};
	\node at (-252.6,-100.5) [vertex] (v14) {};
	\node at (96.4,-245.4) [vertex] (v15) {};
	\node at (28.9,-263.4) [vertex] (v16) {};
	\node at (259.3,36.2) [vertex] (v17) {};
	\node at (-40.3,266.5) [vertex] (v18) {};
	\node at (29.4,266.5) [vertex] (v19) {};
	\node at (-168.2,213.7) [vertex] (v20) {};
	\node at (-107.7,248.5) [vertex] (v21) {};
	\node at (206.5,164) [vertex] (v22) {};
	\node at (96.8,248.3) [vertex] (v23) {};
	\node at (157.2,213.4) [vertex] (v24) {};
	\draw (v1) to (v2);
	\draw (v1) to (v3);
	\draw (v1) to (v4);
	\draw (v1) to (v5);
	\draw (v1) to (v6);
	\draw (v1) to (v7);
	\draw (v1) to (v8);
	\draw (v1) to (v9);
	\draw (v1) to (v10);
	\draw (v1) to (v11);
	\draw (v1) to (v12);
	\draw (v1) to (v13);
	\draw (v1) to (v14);
	\draw (v1) to (v15);
	\draw (v1) to (v16);
	\draw (v1) to (v17);
	\draw (v1) to (v18);
	\draw (v2) to (v3);
	\draw (v2) to (v4);
	\draw (v2) to (v5);
	\draw (v2) to (v6);
	\draw (v2) to (v7);
	\draw (v2) to (v10);
	\draw (v2) to (v11);
	\draw (v2) to (v13);
	\draw (v2) to (v14);
	\draw (v2) to (v15);
	\draw (v2) to (v16);
	\draw (v2) to (v17);
	\draw (v2) to (v18);
	\draw (v2) to (v19);
	\draw (v2) to (v20);
	\draw (v2) to (v21);
	\draw (v3) to (v4);
	\draw (v3) to (v5);
	\draw (v3) to (v6);
	\draw (v3) to (v7);
	\draw (v3) to (v9);
	\draw (v3) to (v10);
	\draw (v3) to (v11);
	\draw (v3) to (v12);
	\draw (v3) to (v13);
	\draw (v3) to (v15);
	\draw (v3) to (v16);
	\draw (v3) to (v17);
	\draw (v3) to (v18);
	\draw (v3) to (v19);
	\draw (v3) to (v22);
	\draw (v4) to (v6);
	\draw (v4) to (v9);
	\draw (v4) to (v10);
	\draw (v4) to (v11);
	\draw (v4) to (v13);
	\draw (v4) to (v14);
	\draw (v4) to (v16);
	\draw (v4) to (v19);
	\draw (v4) to (v20);
	\draw (v4) to (v21);
	\draw (v5) to (v6);
	\draw (v5) to (v7);
	\draw (v5) to (v9);
	\draw (v5) to (v11);
	\draw (v5) to (v12);
	\draw (v5) to (v14);
	\draw (v5) to (v15);
	\draw (v5) to (v16);
	\draw (v5) to (v17);
	\draw (v5) to (v18);
	\draw (v5) to (v19);
	\draw (v5) to (v20);
	\draw (v5) to (v22);
	\draw (v5) to (v23);
	\draw (v6) to (v7);
	\draw (v6) to (v8);
	\draw (v6) to (v10);
	\draw (v6) to (v11);
	\draw (v6) to (v12);
	\draw (v6) to (v13);
	\draw (v6) to (v14);
	\draw (v6) to (v16);
	\draw (v6) to (v17);
	\draw (v6) to (v18);
	\draw (v6) to (v19);
	\draw (v6) to (v20);
	\draw (v6) to (v21);
	\draw (v6) to (v22);
	\draw (v7) to (v8);
	\draw (v7) to (v9);
	\draw (v7) to (v10);
	\draw (v7) to (v11);
	\draw (v7) to (v13);
	\draw (v7) to (v14);
	\draw (v7) to (v15);
	\draw (v7) to (v16);
	\draw (v7) to (v17);
	\draw (v7) to (v18);
	\draw (v7) to (v19);
	\draw (v7) to (v20);
	\draw (v7) to (v22);
	\draw (v7) to (v23);
	\draw (v8) to (v9);
	\draw (v8) to (v10);
	\draw (v8) to (v12);
	\draw (v8) to (v13);
	\draw (v8) to (v14);
	\draw (v8) to (v15);
	\draw (v8) to (v16);
	\draw (v8) to (v17);
	\draw (v8) to (v18);
	\draw (v8) to (v22);
	\draw (v8) to (v23);
	\draw (v9) to (v10);
	\draw (v9) to (v11);
	\draw (v9) to (v12);
	\draw (v9) to (v13);
	\draw (v9) to (v14);
	\draw (v9) to (v15);
	\draw (v9) to (v16);
	\draw (v9) to (v17);
	\draw (v9) to (v18);
	\draw (v9) to (v19);
	\draw (v9) to (v22);
	\draw (v9) to (v23);
	\draw (v9) to (v24);
	\draw (v10) to (v11);
	\draw (v10) to (v13);
	\draw (v10) to (v14);
	\draw (v10) to (v17);
	\draw (v10) to (v18);
	\draw (v10) to (v20);
	\draw (v11) to (v16);
	\draw (v11) to (v17);
	\draw (v11) to (v19);
	\draw (v11) to (v21);
	\draw (v11) to (v22);
	\draw (v12) to (v13);
	\draw (v12) to (v15);
	\draw (v12) to (v16);
	\draw (v12) to (v18);
	\draw (v12) to (v22);
	\draw (v12) to (v23);
	\draw (v12) to (v24);
	\draw (v13) to (v14);
	\draw (v13) to (v15);
	\draw (v13) to (v16);
	\draw (v13) to (v17);
	\draw (v13) to (v18);
	\draw (v13) to (v19);
	\draw (v13) to (v22);
	\draw (v14) to (v15);
	\draw (v14) to (v16);
	\draw (v14) to (v17);
	\draw (v14) to (v18);
	\draw (v14) to (v19);
	\draw (v14) to (v21);
	\draw (v14) to (v24);
	\draw (v15) to (v16);
	\draw (v15) to (v17);
	\draw (v15) to (v18);
	\draw (v15) to (v19);
	\draw (v15) to (v22);
	\draw (v15) to (v23);
	\draw (v15) to (v24);
	\draw (v16) to (v17);
	\draw (v16) to (v19);
	\draw (v16) to (v22);
	\draw (v17) to (v18);
	\draw (v17) to (v19);
	\draw (v17) to (v21);
	\draw (v17) to (v22);
	\draw (v17) to (v24);
	\draw (v18) to (v19);
	\draw (v18) to (v21);
	\draw (v18) to (v24);
	\draw (v19) to (v20);
	\draw (v19) to (v22);
	\draw (v19) to (v23);
	\draw (v20) to (v21);
	\draw (v20) to (v23);
	\draw (v21) to (v24);
	\draw (v22) to (v24);
	\draw (v23) to (v24);
\end{tikzpicture}

%% file: c12comp8a.tex
\begin{tikzpicture}[x=0.1mm,y=0.1mm,thick,vertex/.style={circle,draw,minimum size=10,inner sep=0pt,fill=lightgray}]
	\footnotesize
	\node at (-120,-20) [vertex,label=below:{$B3$}] (v1) {};
	\node at (-240,-20) [vertex,label=below:{$A3$}] (v2) {};
	\node at (0,-20) [vertex,label=below:{$B1$}] (v3) {};
	\node at (-120,100) [vertex,label=above:{$B0$}] (v4) {};
	\node at (-240,100) [vertex,label=above:{$A0$}] (v5) {};
	\node at (120,-20) [vertex,label=below:{$A1$}] (v6) {};
	\node at (0,100) [vertex,label=above:{$B2$}] (v7) {};
	\node at (120,100) [vertex,label=above:{$A2$}] (v8) {};
	\draw [<->,bend left=10] (v1) to (v2);
	\draw [<->,bend right=10] (v1) to (v2);
	\draw [<->,bend left=10] (v1) to (v3);
	\draw [<->,bend right=10] (v1) to (v3);
	\draw [<->,bend left=10] (v1) to (v4);
	\draw [<->,bend right=10] (v1) to (v4);
	\draw [<->,bend left=10] (v2) to (v5);
	\draw [<->,bend right=10] (v2) to (v5);
	\draw [<->,bend left=10] (v3) to (v6);
	\draw [<->,bend right=10] (v3) to (v6);
	\draw [<->,bend left=10] (v3) to (v7);
	\draw [<->,bend right=10] (v3) to (v7);
	\draw [<->,bend left=10] (v4) to (v5);
	\draw [<->,bend right=10] (v4) to (v5);
	\draw [<->,bend left=10] (v4) to (v7);
	\draw [<->,bend right=10] (v4) to (v7);
	\draw [<->,bend left=10] (v6) to (v8);
	\draw [<->,bend right=10] (v6) to (v8);
	\draw [<->,bend left=10] (v7) to (v8);
	\draw [<->,bend right=10] (v7) to (v8);
\end{tikzpicture}

%% file: c12comp8b.tex
\begin{tikzpicture}[x=0.1mm,y=0.1mm,thick,vertex/.style={circle,draw,minimum size=10,inner sep =0pt,fill=lightgray}]
	\footnotesize
	\node at (-54.8,-132.3) [vertex,label=below:{$A5$}] (v1) {};
	\node at (-54.8,132.4) [vertex,label=above:{$B5$}] (v2) {};
	\node at (-132.3,-54.8) [vertex,label=left:{$A7$}] (v3) {};
	\node at (132.3,54.8) [vertex,label=right:{$A6$}] (v4) {};
	\node at (54.8,-132.3) [vertex,label=below:{$B4$}] (v5) {};
	\node at (-132.3,54.8) [vertex,label=left:{$B7$}] (v6) {};
	\node at (54.9,132.4) [vertex,label=above:{$A4$}] (v7) {};
	\node at (132.3,-54.9) [vertex,label=right:{$B6$}] (v8) {};
	\draw [<->] (v1) to (v2);
	\draw [<->,bend left=10] (v1) to (v3);
	\draw [<->,bend right=10] (v1) to (v3);
	\draw [<->,bend left=10] (v1) to (v4);
	\draw [<->,bend right=10] (v1) to (v4);
	\draw [<->] (v1) to (v5);
	\draw [<->,bend left=10] (v2) to (v6);
	\draw [<->,bend right=10] (v2) to (v6);
	\draw [<->] (v2) to (v7);
	\draw [<->] (v3) to (v6);
	\draw [<->,bend left=10] (v3) to (v7);
	\draw [<->,bend right=10] (v3) to (v7);
	\draw [<->] (v3) to (v8);
	\draw [<->] (v4) to (v6);
	\draw [<->,bend left=10] (v4) to (v7);
	\draw [<->,bend right=10] (v4) to (v7);
	\draw [<->] (v4) to (v8);
	\draw [<->] (v5) to (v7);
	\draw [<->,bend left=10] (v5) to (v8);
	\draw [<->,bend right=10] (v5) to (v8);
\end{tikzpicture}